\documentclass[11pt]{article}
\usepackage{amsmath,amsfonts,amssymb, amsthm,enumerate,graphicx,xcolor,url}
\usepackage{mathtools}
\usepackage{tikz,authblk}
\usepackage{subfig}

\newtheorem{theorem} {{\textsf{Theorem}}}
\newtheorem{proposition}[theorem]{{\textsf{Proposition}}}
\newtheorem{corollary}[theorem]{{\textsf{Corollary}}}

\newtheorem{definition}[theorem]{{\textsf{Definition}}}

\newtheorem{lemma}[theorem]{{\textsf{Lemma}}}
\newcommand{\Star}{\mbox{\upshape st}\,}
\newcommand{\lk}{\mbox{\upshape lk}\,}

\begin{document}
\title{A characterization of normal $3$-pseudomanifolds with $g_2\leq4$}
\author{Biplab Basak and Raju Kumar Gupta}
	
\date{}
	
\maketitle
	
\vspace{-12mm}
\begin{center}
	
\noindent {\small Department of Mathematics, Indian Institute of Technology Delhi, New Delhi 110016, India$^{1}$.}
	
\end{center}

\footnotetext[1]{{\em E-mail addresses:} \url{biplab@iitd.ac.in} (B.
	Basak), \url{Raju.Kumar.Gupta@maths.iitd.ac.in} (R. K. Gupta).}

\medskip

\begin{center}
\date{February 14, 2022}
\end{center}

\hrule

\begin{abstract}
We characterize normal $3$-pseudomanifolds with $g_2\leq4$. We know that if a $3$-pseudomanifold  with $g_2\leq4$ does not have any singular vertices then it is a $3$-sphere.  We first prove that a normal $3$-pseudomanifold with $g_2\leq4$ has at most two singular vertices.
Then we prove that a normal $3$-pseudomanifold with $g_2 \leq 4$, which is not a $3$-sphere is obtained from some boundary of $4$-simplices  by a sequence of operations  connected sum, edge expansion and an edge folding. In addition, by using \cite{Walkup}, we re-framed the characterization of normal $3$-pseudomanifolds with $g_2\leq 9$, when it has no singular vertices.
\end{abstract}

\noindent {\small {\em MSC 2020\,:} Primary 05E45; secondary 52B05, 57Q05, 57Q25, 57Q15.
	
	}

\medskip

\section{Introduction}
For every $d$-dimensional simplicial complex $K$, the $f$-vector is defined as $(f_0,f_1, \dots ,f_d)$, where each $f_i$ is number of $i$-dimensional faces of $K$. This vector has a remarkable importance in simplicial topology, in fact it is one of the main tools to study simplicial complexes. Walkup \cite{Walkup} proved that the number $f_1 - 4f_0 + 10$ is non-negative for any closed connected $3$-manifold $K$, and later this number is called as $g_2$. Walkup also proved that for a $3$-manifold $K$,  $g_2(K)=0$ if and only if $K$ is  triangulation of a stacked sphere. Later, Barnette \cite{Barnette1,Barnette2,Barnette3}  proved $g_2(K)\geq 0$ for boundary complex of a simplicial $(d + 1)$-polytope, and Kalai \cite{Kalai} proved the same result for normal pseudomanifolds of dimension at least 3 where the link of each face of codimension 2 was stacked sphere. Kalai used a new approach of rigidity theory and started studying the  simplicial complexes with their $g$-vector. Kalai proved that $g_2(K)\geq g_2(lk(v))$ for any vertex $v\in K$ where $K$ is $d$-rigid pure $d$-simplicial complex. Later, Fogalsanger \cite{Fogelsanger} proved the same result for any normal $d$-pseudomanifold. In [\cite{Gromov}, pages $211-212$], Gromov has similar result on the nonnegativity of $g_2$. In \cite{Swartz2008}, we can find that there are only a finite number of PL-homeomorphism classes of combinatorial manifolds of a given dimension $d$ and a given upper bound on $g_2$. In \cite{BagchiDatta98}, we can find all possible classifications of a triangulated $d$-pseudomanifold with at most $d+4$ vertices. Also, We can find the similar results in [\cite{NovikSwartz},\cite{Swartz2009},\cite{TayWhiteWhiteley}]. Further, 
 using the rigidity theory the complete characterization came for a normal $d$-pseudomanifold for $g_2(K)\leq2$ where in every case $K$ is a polytopal spehere. These cases for $g_2(K)=0,1,2$ are settled by Kalai \cite{Kalai}, Nevo and Novinsky \cite{NevoNovinsky} and Zheng \cite{Zheng} respectively. Later, Basak and Swartz \cite{BasakSwartz} gave the complete  characterization for normal $3$-pseudomanifolds $(|K|\not\cong\mathbb{S}^3)$ with $g_2(K)=3$.
 
In this paper we give the complete characterization for a normal $3$-pseudomanifold with $g_2(K)\leq 4$. 
We first prove that a normal $3$-pseudomanifold with $g_2\leq 4$ can have at most two singular vertices. We prove that, if $K$ has no singular vertices and $g_2\leq 9$ then $K$ is obtained from some boundary of $4$-simplices by a sequence of operations connected sum,  bistellar $1$-move, edge contraction and edge expansion. The characterization of a normal $3$-pseudomanifold $K$ with $g_2\leq 3$ is known from \cite{Kalai, NevoNovinsky, Zheng,BasakSwartz}. We prove that, if $K$ has some singular vertices and $g_2=4$ then $K$ is obtained from some boundary of $4$-simplices  by a sequence of operations connected sum, edge expansion and an edge folding.

\section{Preliminaries}
A polytope is a convex hull of finite number of points. A $d$-simplex is defined as a $d$ dimensional polytope of exactly $d+1$ points. For example a point is a $0$-simplex, a line segment between two points is a $1$-simplex, similarly a triangle is a $2$-simplex and so on.
Let $\sigma$ be a simplex. Then we denote $V(\sigma)$ as the vertex set of $\sigma$ and face of $\sigma$ is defined as the convex hull of non-empty subset of $V(\sigma)$. We denote a face $\tau$ of $\sigma$ as $\tau\leq\sigma$.    
A simplicial complex is defined as finite collection of simplices where for every simplex $\sigma\in K$, all the  faces of $\sigma$ are in $K$ and for any two simplices $\sigma$ and $\tau$ in $K$, $\sigma\cap\tau$ is either empty or it is face of both the simplices $\sigma$ and $\tau$. If $K$ is a simplicial complex then $|K|:= \cup_{\sigma \in K} \sigma$ is a compact polyhedron and is called the geometric carrier of $K$ or the underlying polyhedron corresponding to $K$, and we say $K$ is a triangulation of $|K|$.
The dimension of a face $\sigma\in K$ is $|\sigma|-1$ and dimension of $K$ is defined as the maximum of dimension of simplices in $K$. If $\sigma$ and $\tau$ are two simplices of $K$ then their join $\sigma\star \tau$ is defined as the set $\{\lambda a + \mu b : a\in\sigma ,b\in\tau; \lambda, \mu \in [0, 1]$ and $ \lambda + \mu = 1\}$. Similarly we define join of two simplicial complexes $K_1$ and $K_2$ as $\{\sigma \star \tau : \sigma\in K_1, \tau\in K_2\}$.
The link of a face $\sigma$ in $K$ is the collection of faces in $K$ which do not intersect $\sigma$ and whose join with $\sigma$ lies in $K$ and is denoted by $\lk (\sigma, K)$. Star of a face $\sigma$ in $K$ is  $\{\gamma : \gamma\leq \sigma\star\alpha$ and $ \alpha\in \lk (\sigma, K)\}$ and is denoted by $\Star (\sigma, K)$. The graph of the simplicial complex $K$ is denoted by $G(K)$.
 The simplicial complex $K$ is called a normal $d$-pseudomanifold without boundary if all the facets (maximal faces) of $K$ are of same dimension, every face of dimension $d-1$ is contained in exactly two facets and link of every face of co-dimension $2$ or more is connected.  If the face of dimension $d-1$ is contained in one facets then $K$ is normal $d$-pseudomanifold with boundary. For a normal $3$-pseudomanifold $K$, link of a vertex $v$ is a closed connected surface say $S$. If $S\cong \mathbb{S}^2$ then we call $v$ is a non-singular vertex; otherwise, $v$ is called a singular vertex. Now, we have the following definitions and lemmas. 

\begin{definition}[Edge contraction]\label{Edge contraction} Let $K$ be a normal $d$-pseudomanifold and $u,v$ be two vertices of $K$ such that $uv \in K$ and $\lk (u,K)\cap \lk (v,K)=\lk (uv,K)$. Let $K'= K \setminus (\{\alpha\in K : u\leq \alpha\} \cup \{\beta\in K : v\leq \beta\})$, then $K_1=K'\cup\{w\star \partial (K')\}$ for some vertex $w$ is said to be obtained from $K$ by contracting the edge $uv$ and this process is called the edge contraction.
\end{definition}
 \begin{lemma}[\cite{BSR}] \label{homeomorphic}
 	For $d\geq 3$, let $K$ be a normal $d$-pseudomanifold. Let $uv$ be an edge of $K$ such that $\lk (u,K)\cap \lk (v,K)= \lk (uv,K)$ and $|\lk (v,K)| \cong \mathbb{S}^{d-1}$. If $K_1$ is the normal pseudomanifold obtained from $K$ by contracting the edge $uv$ then  $|K|\cong |K_1|$.
 \end{lemma}
 
\begin{definition}[Handle addition and Connected sum]
	Let $\sigma_1$ and $\sigma_2$ be two facets of a simplicial complex $\Delta$.  A bijective map $\psi:\sigma_1 \to \sigma_2$ is admissible if  for all  vertices $x$ of $\sigma_1$ the length of path from $x$ to $\psi(x)$ is at least three (cf. \cite{BagchiDatta98}). Now from the map $\psi$ we can form the complex $\Delta^\psi$ obtained by identifying all faces $\rho_1 \subseteq \sigma_1 $ and $\rho_2 \subseteq \sigma_2 $ such that $\psi(\rho_1) = \rho_2,$ and then removing the facet formed by identifying $\sigma_1$ and $\sigma_2.$  In this case we call $\Delta^\psi$ is obtained from  $\Delta$ by {\bf handle addition}  if $\sigma_1$ and $\sigma_2$ are from same connected component of $\Delta$ and is obtained by {\bf connected sum} if $\sigma_1$ and $\sigma_2$ are from distinct components.
\end{definition}
A straightforward computation shows that for a $d$-dimensional complex $\Delta$ handle additions satisfy,
\begin{equation} \label{g_2: handles}
	g_2(\Delta^\psi) = g_2(\Delta) + \binom{d+2}{2}.
\end{equation}

\noindent Similarly, for connected sum
\begin{equation} \label{g_2:connected sum}
	g_2(\Delta_1 ~\#_\psi~ \Delta_2) = g_2(\Delta_1) + g_2(\Delta_2).
\end{equation}

\begin{lemma}[\cite{BasakSwartz}] \label{lemma:missingtetra1}
	Let $\Delta$ be a normal three-dimensional pseudomanifold and let $\tau$ be a missing tetrahedron in $\Delta.$  If for every vertex $x \in \tau$  the missing triangle formed by the other three vertices separates the link of $x,$  then $\Delta$ was formed using handle addition or connected sum.  
\end{lemma}

\begin{definition}[Edge folding \cite{BasakSwartz}]
	Let $\sigma_1$ and $\sigma_2$ be two facets of a simplicial complex $\Delta$ whose intersection is an edge $uv$. A bijection $\psi : \sigma_1 \to \sigma_2$ is \textbf{ edge folding admissible} if $\psi(u)= u, \psi(v) = v$ and for all other vertices $y$ of $\sigma_1$, all paths of length two or less from $y$ to $\psi(y)$ go through either $u$ or $v$. As before, identify all faces $\rho_1\subseteq\sigma_1$ and $\rho_2 \subseteq \sigma_2$ such that $\psi : \rho_1\to \rho_2$ is a bijection. The complex obtained by removing the facet resulting from identifying $\sigma_1$ and $\sigma_2$ is denoted by $\Delta^\psi_{uv}$ and is called an \textbf{ edge folding} of $\Delta$ at $uv$. Further, $\Delta$ is an \textbf{ edge unfolding} of $\Delta^\psi_{uv}$.
\end{definition}
\noindent If $\Delta$ is a normal $d$-pseudomanifold and $\Delta^\psi_{uv}$ is obtained from $\Delta$ by an edge folding at $uv$, then 
\begin{equation} \label{edge folding g2}
	g_2(\Delta^\psi_{uv}) = g_2(\Delta)+\binom{d}{2}.
\end{equation}
A \textbf{missing triangle} of $\Delta$ is a triangle $abc$ such that $abc\notin \Delta$ but $\partial(abc)\in\Delta$. Similarly a \textbf{missing tetrahedron} of $\Delta$ is a tetrahedron $abcd$ such that $abcd\notin\Delta$ but $\partial(abcd)\in\Delta$.
Let $\sigma=abcv$ be a missing tetrahedron of $\Delta$, where $v$ is a vertex of $\sigma$. Then the triangle $abc$ is a missing triangle in $\lk (v,\Delta)$, and is denoted by $\sigma-v$. Further, A small neighborhood of $|\partial(abc)|$ in $|\lk (v,\Delta)|$ is an annulus if  $|\lk (v,\Delta)|$ is an orientable surface and a small neighborhood of $|\partial(abc)|$ in $| \lk (v,\Delta)|$ is either an annulus or a M\"{o}bius strip If $|\lk (v,\Delta)|$ is a non-orientable surface.

\begin{lemma}[\cite{BasakSwartz}]\label{lemma:missingtetra2}
	Let $\Delta$ be a 3-dimensional normal pseudomanifold. Let $\tau=abuv$ be a missing facet in $\Delta$ such that $(i)$ for $x\in\{a,b\}$, $\partial(\tau -x)$ separates  $\lk (x,\Delta),$ and $(ii)$ a small neighborhood of $|\partial(abv)|$ in $|\lk (u,\Delta)|$ is a M\"{o}bius strip. Then a small neighborhood of $|\partial(abu)|$ in $|\lk (v,\Delta)|$ is also a M\"{o}bius strip. Further, there exists $\Delta'$ a three-dimensional normal pseudomanifold such that  $\Delta = (\Delta')^\psi_{uv}$ is obtained from an edge folding at $uv \in \Delta'$  and $abuv$ is the removed facet.
\end{lemma}

\section{Rigidity Theory:}
In this section we give a short introduction of rigidity theory. Rigidity has became a direction for analyzing $f$-vectors. Here, we have discussed some results (cf. \cite{BasakSwartz, Kalai, NevoNovinsky, TayWhiteWhiteley, Zheng}) that we have used.

Let  $G=(V,E)$ be a simple graph. Let $g : V\rightarrow R^d$ be a function. Then $(G, g)$ is called rigid if there exists $\epsilon >0$ such that if $h : V\rightarrow R^d$ with  $||g(v) - h(v)|| < \epsilon$ for all $v\in V$ and $||g(v) - g(u)|| = ||h(v) - h(u)||$ for all $uv\in E$, then $||g(u) - g(v)|| = ||h(u) - h(v)||$ for all $u, v\in V$. 

The graph $G$ is called generically $d$-rigid if for generic choices of $f : V\rightarrow R^d$, $(G, f)$ is rigid. A simplicial complex $K$ is called $d$-rigid if $G(K)$ is generically $d$-rigid. 
Note that it is immediate from the definition that if $G$ is generically $d$-rigid and $e$ is an
edge whose vertices are in $G$, but $e\not \in G$, then $G\cup \{e\}$ is generically $d$-rigid.  Further, from  \cite{Fogelsanger, Kalai}, we know the following. 

%
%
\begin{theorem} [\cite{Fogelsanger, Kalai}] \label{lemma:d+1 rigid}
Every Normal $d$-pseudomanifold is $(d + 1)$-rigid if $d\geq 2$.
\end{theorem}
%
%
From \cite{NevoNovinsky, TayWhiteWhiteley}, we have the following cone lemma.

\begin{lemma}[\cite{NevoNovinsky, TayWhiteWhiteley}]\label{lemma:cone lemma} Let $K$ be $d$-rigid and let $C(K)$ be the cone of $K$
with cone vertex $v$. Then $C(K)$ is $(d + 1)$-rigid. 
\end{lemma}

Further, we have the following lower bounds for $g_2$ due to Kalai (cf. \cite{Kalai}).

\begin{lemma}[\cite{Kalai}]\label{lemma:rigid subcomplex}
	 Let $K$ be $(d + 1)$-rigid $d$-dimensional normal pseudomanifold and $\sigma$ be a $(d + 1)$-rigid subcomplex of $K$. Then $g_2(\sigma)\leq g_2(K)$.
\end{lemma}

\begin{lemma}[\cite{Kalai}]\label{lemma:lower bound}
If $K$ is a normal $3$-pseudomanifold and $v\in K$ is a vertex  then $g_2(K)\geq g_2(\Star (v)) = g_2(\lk (v))$.
\end{lemma}

Let $K$ be a normal $3$-pseudomanifold and $v\in K$ be a vertex of $K$. Then $\lk (v)$ is a normal 2-pseudomanifold, and hence by Lemma \ref{lemma:d+1 rigid}, $\lk (v)$ is $3$-rigid. Further, from Lemma \ref{lemma:cone lemma}, we have $\Star (v)$ is 4-rigid, i.e.,  $G(\Star (v))$ is 4-rigid.
Let $e_1,e_2, \dots ,e_n$ be the edges in $K$ such that  whose vertices are in $G(\Star (v))$, but the edges $e_1,e_2, \dots ,e_n$ are not in $G(\Star (v))$. Then $G(\Star (v))\cup \{e_1,e_2, \dots ,e_n\}$ is also 4-rigid. Since $G(\Star (v)\cup \{e_1,e_2, \dots ,e_n\})=G(\Star (v))\cup \{e_1,e_2, \dots ,e_n\}$, the subcomplex  $\Star (v)\cup \{e_1,e_2, \dots ,e_n\}$ is also 4-rigid. Therefore, by Lemma \ref{lemma:rigid subcomplex}, we have $g_2(K)\geq g_2(\Star (v)\cup \{e_1,e_2, \dots ,e_n\})=g_2(\Star (v))+n= g_2(\lk (v))+n$. Thus, we have the following result.

\begin{corollary}\label{missingedges}
Let $K$ be a normal $3$-pseudomanifold and $v\in K$ be a vertex in $K$. Let $e_1,e_2, \dots ,e_n$ be the edges in $K$ such that  the vertices of  $e_1,e_2, \dots ,e_n$ are in $\lk (v)$, but the edges $e_1,e_2, \dots ,e_n$ are not in $\lk (v)$. Then $g_2(K)\geq g_2(\lk (v))+n$.
\end{corollary}

\section{Normal $3$-pseudomanifolds with no singular vertices}

Let $K$ be a normal $3$-pseudomanifold with $g_2(K)\leq 9$ and $K$ has no singular vertices. It is clear from \cite{Walkup} that $K$ is a $3$-sphere. In this section we give the complete characterization of such $K$. We know that  $g_2(K)=0$ implies $K$ is a stacked sphere, i.e., $K$ is a connected sum of boundary of  4-simplices (or $K$ is obtained from a boundary of a 4-simplex by facets subdivisions.) Thus we assume that $1\leq g_{2}(K)\leq 9$. 

If $K$ has a missing tetrahedron $\sigma$ then by Lemma \ref{lemma:missingtetra1}, $K$ was formed using handle addition or connected sum.  If $K$ was formed using handle addition then $g_2(K)\geq 10$, which is a contradiction. Thus,  $K=K_1\#K_2$ where $\sigma\in K_1$ as well as $\sigma\in K_2$. So, there is a finite collection of normal $3$-pseudomanifolds $K_1,\dots,K_n$ such that $K=K_1\#\cdots\# K_n$ where each of $K_i$ has no missing tetrahedron and $g_2(K_i)\leq 9$ for all $1\leq i\leq n$. Then we can proceed with $K_i$. So, without loss of generality we can assume that $K$ has no missing tetrahedron.

Let $uv\in K$ be an edge and $d(uv)=3$. If $\lk (uv)=\partial(abc)$. If possible, let $abc \in K$. Then $\partial(uabc)$ and $\partial(vabc) \in K$. Then $uabc, vabc\in K$ and hence $uabc$, $vabc$, $uvab$, $uvbc$, $uvac$ are all in $K$, i.e., $\partial(uvabc) \subset K$. But this is possible only if $\partial(uvabc) = K$ which is a contradiction as $K$ is a manifold with $g_2(K)\geq 1$. Thus, $abc \not \in K$. Let $K'=(K-\{\alpha\in K: uv \leq \alpha\})\cup\{abc, uabc, vabc\}$. Since $abc\not \in K$, it is easy to see that $|K'|\cong |K|$, $f_0(K')=f_0(K)$ and $f_1(K')=f_1(K)-1$ and we get $g_2(K')=g_2(K)-1$. This combinatorial operation is called a bistellar $2$-move.  Note that $K'$ may have a missing tetrahedron. The reverse of this operation is the following

\smallskip

\noindent{\bf{Bistellar $1$-move:}} Let $K$ be a normal $3$-pseudomanifold  (may have  a missing tetrahedron) and $abc\in K$ is a $2$-simplex. Let $\lk (abc)=\{u,v\}$ where $u,v$ are two vertices in $K$ and $uv\not\in K$. Let $K'=(K-\{abc,uabc,vabc\})\cup\{uvab,uvbc,uvac\}$. It is easy to see that $|K'|\cong |K|$, $f_0(K')=f_0(K)$ and $f_1(K')=f_1(K)+1$ and we get $g_2(K')=g_2(K)+1$.

\smallskip

Let $uv\in K$ and $d(uv)=n$. If $\lk (u)\cap \lk (v)-\lk (uv)$ is an empty set then contract the edge $uv$ (cf. Definition \ref{Edge contraction}) and we get a new complex $K'$ such that $|K|\cong |K'|$ with $g_2(K')=g_2(K)-(n-3)$. If $n\geq 4$ then $g_2(K')\leq g_2(K)-1$.
This combinatorial operation is called edge contraction.   Note that $K'$ may have  a missing tetrahedron.   Now, we define the reverse of this operation.

\smallskip

\noindent{\bf{Edge expansion:}} Let $K$ be a normal $3$-pseudomanifold (may have  a missing tetrahedron). Let $w\in K$ and $C_n(u_1,u_2, \dots ,u_n)$ be an $n$-cycle in $\lk (w)$. Then we retriangulate the  $3$-ball $\Star w$ by removing $w$ and inserting $n$ tetrahedra $uvu_{1}u_2,uvu_{2}u_3, \dots ,uvu_{n}u_1$ and then coning off the two hemispheres formed by the $\lk (w)$ and the circle $C_n(u_1,u_2, \dots ,u_n)$ by $u$ and $v$ respectively. Let $K'$ be the  resulting complex then $g_2(K')=g_2(K)+n-3$ and $|K|'\cong |K|$.
If $d(uv)\geq 4$, i.e.,  $n\geq 4$  then $g_2(K') \geq g_2(K)+1$. If there is some $x \in \lk (w)$ such that $\lk (xw)=C_n(u_1,u_2, \dots ,u_n)$ then we say $K'$ is obtained from $K$ by {\em central retriangulation} of $\Star (xw)$ with center $u$ and renaming $w$ by $v$. Thus central retriangulation is a special case of edge expansion.

\smallskip

 Let $w\in K$  be a vertex and $\partial (abc) \in \lk (w)$. If $abc\not \in K$ then we retriangulate the  $3$-ball $\Star w$ by removing $w$ and inserting $abc$ and then coning off the two spheres formed by the link of $w$ and $abc$. Let $K'$ be the  resulting complex then $g_2(K')=g_2(K)-1$ and $|K|'\cong |K|$. We call this operation as `two facets insertion'.   Note that $K'$ may have a missing tetrahedron. Now, we define the reverse of this operation.
 
 \smallskip
  
\noindent{\bf{Two facets contraction:}} Let $K$ be a normal $3$-pseudomanifold (may have  a missing tetrahedron). Let $u$ and $v$ be two vertices in $K$ such that $uv\not\in K$ and $\Star (u)\cap \Star (v)=abc$. Then we retriangulate the 3-ball $\Star (u)\cup \Star (v)$ by removing $u,v,abc$ and coning off the boundary of $\Star (u)\cup \Star (v)$. Let $K'$ be the  resulting complex then $g_2(K')=g_2(K)+1$ and $|K'|\cong |K|$. Note that the combinatorial operation `two facets contraction' is just a combination of a bistellar $1$-move and an edge contraction.

\begin{proposition}\cite[Lemma $10.8$]{Walkup}\label{lemma:g_2<10}
Let $K$ be a closed connected 3-manifold such that $(i)$ $K$ is not a boundary of a $4$-simplex, $(ii)$ $K$ has no missing tetrahedron and $(iii)$ there is no $K'$  with  $|K'|\cong |K|$ and $g_2(K')<g_2(K)$, where $K'$ is obtained from $K$ by  bistellar $2$-move,  edge contraction or two facets insertion. Then $f_1(K) > 4 f_0(K)$, i.e., $g_2(K) > 10$.
\end{proposition}

\begin{theorem}
	If $K$ is a normal $3$-pseudomanifold with $g_2(K)\leq 9$ and $K$ has no singular vertices then $K$ is a $3$-sphere and is obtained from some boundary of $4$-simplices by a sequence of operations connected sum,  bistellar $1$-move, edge contraction and edge expansion. 
\end{theorem}

\begin{proof}
If $g_2(K)=0$ then $K$ is a stacked sphere, i.e., $K$ is a connected sum of boundary of  4-simplices.  So, we assume that $1\leq g_{2}(K)\leq 9$, i.e., $K$ is not a boundary of a $4$-simplex. Let  $K=K_1\#\cdots\# K_n$ where each of $K_i$ has no missing tetrahedron and $g_2(K_i)\leq 9$ for all $1\leq i\leq n$. If $K_i$ is not a  boundary of a $4$-simplex, and there is no $K_i'$  with  $|K_i'|\cong |K_i|$, where $K_i'$ is obtained from $K_i$ by  bistellar $2$-move,  edge contraction or two facets insertion, then by Proposition \ref{lemma:g_2<10} we have $g_2(K_i)>10$. Thus, $g_2(K)>10$, which is a contradiction. Therefore, either $K_i$ is a boundary of a $4$-simplex or there is a $K_i'$  with  $|K_i'|\cong |K_i|$, where $K_i'$ is obtained from $K_i$ by  bistellar $2$-move,  edge contraction or two facets insertion. In the second case, $g_2(K_i')\leq g_2(K_i)-1$, and  $K_i'$ may have a missing tetrahedron. Thus we can repeat the same arguments for each $K_i'$ if $K_i$ is not  a boundary of a $4$-simplex until we get the reduced complex is a connected sum of boundary of a $4$-simplices. Thus, $K$ is obtained from boundary of $4$-simplices by the sequence of operations  connected sum, bistellar $1$-move, edge expansion and two facets contraction. Since `two facets contraction' is just a combination of a bistellar $1$-move and an edge contraction, we can say that $K$ is obtained from boundary of $4$-simplices by the sequence of operations  connected sum, bistellar $1$-move, edge contraction and edge expansion. 
\end{proof}

\begin{lemma}\label{lemma:g2<3}
	If $K$ is a normal $3$-pseudomanifold with $g_2(K)\leq 2$  then $K$ is a $3$-sphere and is obtained from some boundary of $4$-simplices by a sequence of operations connected sum and edge expansion.

\end{lemma}
\begin{proof}
If $g_2(K)=0$ then $K$ is a stacked sphere, i.e., $K$ is a connected sum of boundary of  4-simplices. 
Let $1\leq g_2(K)\leq 2$. We know that $K$ is a 3-sphere. If $K$ has a missing tetrahedron $\sigma$ then $K=K_1\#K_2$ where $\sigma\in K_1$ as well as $\sigma\in K_2$. So, there is a finite collection of normal $3$-pseudomanifolds $K_1, \dots ,K_n$ such that $K=K_1\# \cdots \# K_n$ where each of $K_i$ has no missing tetrahedron and $1\leq g_2(K_i)\leq 2$ for all $1\leq i\leq n$. Now, we can proceed with  each $K_i$ separately if $K_i$ is not a  boundary of  a 4-simplex.
 
Let $v\in K_i$ be a vertex such that  $d(v)=4$. Then link of $v$ is the boundary of a $3$-simplex, say $abcd$. If $abcd$ is a face of $K_i$, then the boundary
of the four-simplex $vabcd$ is in $K_i$ and hence $K_i$ is the boundary of the $4$-simplex, which is not possible. Thus, $abcd$ is a missing tetrahedron in $K_i$, which is again a contradiction. Thus, $d(v)\geq 5$ for every vertex $v\in K_i$. 

Let $u$ be a vertex in $K_i$ then $d(u)\geq 5$. If $\lk (u)\cap \lk (x) - \lk (ux)$ contains an open edge $(a,b)$ for a vertex $x\in \lk (u)$ then $\partial(uxab)$ is a missing tetrahedran, which is a contradiction. If $\lk (u)\cap \lk (x) - \lk (ux)$ contains some vertex, say $y$ for every vertex $x \in \lk (u)$ then $xy\not\in \lk (u)$. Since $d(u)\geq 5$,  there are at least three such edges in $K - \Star (u)$, whose end points are in $\lk (u)$. Then by Corollary \ref{missingedges}, we have $g_2(K_i) \geq 3$. This is a contradiction. Thus, there must be a vertex $v\in \lk (u)$ such that $\lk (u) \cap \lk (v) - \lk (uv)$ is empty. Then we contract the edge $uv$ and we get a new normal $3$-pseudomanifold $\tilde K_i$ such that $|\tilde K_i|\cong |K_i|$. If $d(uv)\geq 4$ then $g_2(\tilde K_i) \leq g_2(K_i)-1$. Here $\tilde K_i$ may have a missing tetrahedron. If $g_2(\tilde K_i) =0$ then we are done, and we can repeat the same arguments from the beginning if $g_2(\tilde K_i) = 1$. If $d(uv)=3$ then  $g_2(\tilde K_i)=g_2(K_i)$ but the number of vertices is reduced. Since $K$ had finite number of vertices, proceeding in this way after finite number of steps we must have a normal $3$-pseudomanifold $\bar{K}_i$ such that for some edge $pq\in\bar{K_i}$, $\lk (p) \cap \lk (q) - \lk (pq)$ is empty and $d(pq)\geq 4$. Then we are done. Thus,  $K$ is obtained from boundary of $4$-simplices by the sequence of operations connected sum and edge expansions.
\end{proof}

\section{Normal $3$-pseudomanifolds with singular vertices}

Let $K$ be a normal $3$-pseudomanifold  with singular vertices and $g_2(K)\leq 4$. By Lemma \ref{lemma:lower bound} the only possible link of the singular vertices in $K$ is $\mathbb{RP}^2$ and  $3\leq g_2(K)\leq 4$. From now onwards for a singular vertex in $K$, we mean there is a vertex $v\in K$ such that $|\lk (v)|\cong \mathbb{RP}^2$. We know that the sum of the Euler characteristic of link of the vertices for a $3$-dimensional normal pseudomanifold is even. Thus, the total number of singular vertices  in $K$ is even. The sum $g_2 + g_3$ in a $3$-dimensional normal pseudomanifold is $$\sum_{v\in V(K)}(2-\chi(\lk (v))).$$
From \cite{BasakSwartz}, we know that $g_2(K)=3$ implies $K$ has exactly two singular vertices. From [\cite{Swartzcounting} Corollary $1.8$], we have $g_2(K)=4$ implies that $g_3(K)\leq 5$. Since the Euler characteristic of a singular vertex is one in our case, the number of singular vertices in $K$ is less than $10$. The possible number of singular vertices in $K$ is $0,2,4,6$ and $8$. But, if $K$ has $6$ singular vertices then Theorem $2.12$ of \cite{Swartzcounting} implies $g_2(K)\geq 5$, which is a contradiction. Thus the case of $6$ singular vertices in $K$ is not possible.

We have defined edge contraction in Definition \ref{Edge contraction}. Let $uv$ be an edge of $K$ such that $\lk (u,K)\cap \lk (v,K)= \lk (uv,K)$ and $v$ is non-singular. If $K_1$ is the normal pseudomanifold obtained from $K$ by contracting the edge $uv$ then from Lemma \ref{homeomorphic}, we know that  $|K|\cong |K_1|$. We can define the reverse operation of this as well. 

Let $K'$ be a normal $3$-pseudomanifold with $g_2(K)\leq 4$. Let $w\in K'$ and $C_n(u_1,u_2, \dots ,u_n)$ be an $n$-cycle in $ \lk (w)$ such that $C_n(u_1,u_2, \dots ,u_n)$ separates $ \lk (w)$. Then one portion will be a disc and other portion will be either a disc or a M\"{o}bius strip. Then we retriangulate $\Star w$ by removing $w$ and inserting $n$ tetrahedra $uvu_{1}u_2,uvu_{2}u_3, \dots ,uvu_{n}u_1$ and then coning off the two portions formed by the $\lk (w)$ and the circle $C_n(u_1,u_2, \dots ,u_n)$ by $u$ and $v$ respectively. Let $K$ be the  resulting complex then $g_2(K)=g_2(K')+n-3$. Since $K'$ is formed from $K$ by contracting the edge $uv$, where one vertex (which was coned over the disc) is non-singular, we have  Lemma \ref{homeomorphic},  $|K'|\cong |K|$.

\begin{lemma}\label{lemma:three mobius}
	Let $\Delta$ be a 3-dimensional normal pseudomanifold. Let $\tau=abcd$ be a missing tetrahedron such that for $x\in\{a,b, c\}$,  a small neighborhood of  $|\partial(\tau \setminus\{x\}])|$ in $|\lk (x,\Delta)|$ is a M\"{o}bius strip. Then a small neighborhood of  $|\partial(abc)|$ in $|\lk (d,\Delta)|$ is also  a M\"{o}bius strip.
\end{lemma}

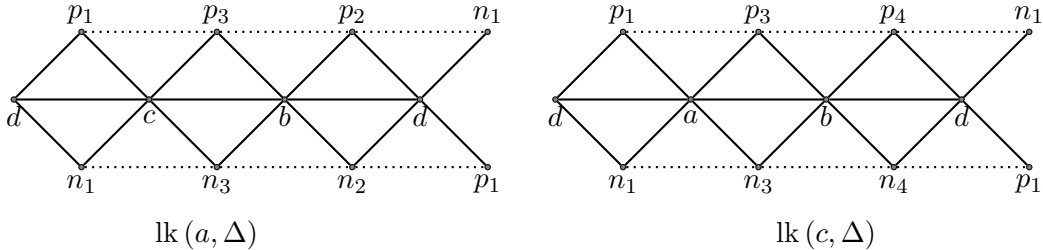
\begin{figure}[ht]
	\tikzstyle{ver}=[]
	\tikzstyle{vertex}=[circle, draw, fill=black!50, inner sep=0pt, minimum width=2pt]
	\tikzstyle{edge} = [draw,thick,-]
	\centering
	\begin{tikzpicture}[scale=0.45]
		
		\begin{scope}[shift={(-8,0)}]
			\foreach \x/\y/\z in {-6/1/b,-2/1/c,2/1/a,6/1/b1}{
				\node[vertex] (\z) at (\x,\y){};
			}
			
			\foreach \x/\y/\z in {-6/0.5/d,-2/0.5/c,2/0.5/b,6/0.5/d}{
				\node[ver] () at (\x,\y){$\z$};
			}
			
			\foreach \x/\y/\z in {-4/3/p_2,0/3/p_4,4/3/p_1,8/3/p2}{
				\node[vertex] (\z) at (\x,\y){};
			}
			
			\foreach \x/\y/\z in {-4/3.5/p_1,0/3.5/p_3,4/3.5/p_2,8/3.5/n_1}{
				\node[ver] () at (\x,\y){$\z$};
			}
			
			\foreach \x/\y/\z in {-4/-1/n_2,0/-1/n_4,4/-1/n_1,8/-1/q2}{
				\node[vertex] (\z) at (\x,\y){};
			}
			
			\foreach \x/\y/\z in {-4/-1.5/n_1,0/-1.5/n_3,4/-1.5/n_2,8/-1.5/p_1}{
				\node[ver] () at (\x,\y){$\z$};
			}

			\foreach \x/\y in {b/c,c/a,a/b1,b/p_2,c/p_2,b/n_2,c/n_2,c/p_4,a/p_4,c/n_4,a/n_4,a/p_1,b1/p_1,a/n_1,b1/n_1,
				b1/p2,b1/q2}{
				\path[edge] (\x) -- (\y);}
			
			\foreach \x/\y in {p_2/p_4,p_1/p2,n_2/n_4,n_1/q2}{
				\path[edge, dotted] (\x) -- (\y);}
			\path[edge, dotted] (p_4) -- (p_1);
			\path[edge, dotted] (n_4) -- (n_1);
			
		\end{scope}
		
		\begin{scope}[shift={(8,0)}]
			\foreach \x/\y/\z in {-6/1/b,-2/1/c,2/1/d,6/1/b1}{
				\node[vertex] (\z) at (\x,\y){};
			}
			
			\foreach \x/\y/\z in {-6/0.5/d,-2/0.5/a,2/0.5/b,6/0.5/d}{
				\node[ver] () at (\x,\y){$\z$};
			}
			
			\foreach \x/\y/\z in {-4/3/p_2,0/3/p_4,4/3/p_1,8/3/p2}{
				\node[vertex] (\z) at (\x,\y){};
			}
			
			\foreach \x/\y/\z in {-4/3.5/p_1,0/3.5/p_3,4/3.5/p_4,8/3.5/n_1}{
				\node[ver] () at (\x,\y){$\z$};
			}
			
			\foreach \x/\y/\z in {-4/-1/n_2,0/-1/n_4,4/-1/n_1,8/-1/q2}{
				\node[vertex] (\z) at (\x,\y){};
			}
			
			\foreach \x/\y/\z in {-4/-1.5/n_1,0/-1.5/n_3,4/-1.5/n_4,8/-1.5/p_1}{
				\node[ver] () at (\x,\y){$\z$};
			}

			\foreach \x/\y in {b/c,c/d,d/b1,b/p_2,c/p_2,b/n_2,c/n_2,c/p_4,d/p_4,c/n_4,d/n_4,d/p_1,b1/p_1,d/n_1,b1/n_1,
				b1/p2,b1/q2}{
				\path[edge] (\x) -- (\y);}
			
			\foreach \x/\y in {p_2/p_4,p_1/p2,n_2/n_4,n_1/q2}{
				\path[edge, dotted] (\x) -- (\y);}
			\path[edge, dotted] (p_4) -- (p_1);
			\path[edge, dotted] (n_4) -- (n_1);
			
		\end{scope}
		
		\node[ver] () at (-8.3,-3){$\lk (a,\Delta)$};
		\node[ver] () at (10,-3){$\lk (c,\Delta)$};
		
	\end{tikzpicture}
	\caption{$\lk (a,\Delta)$ and $\lk (c,\Delta)$ in $\Delta$.}\label{fig:lk(a&c)}
\end{figure}

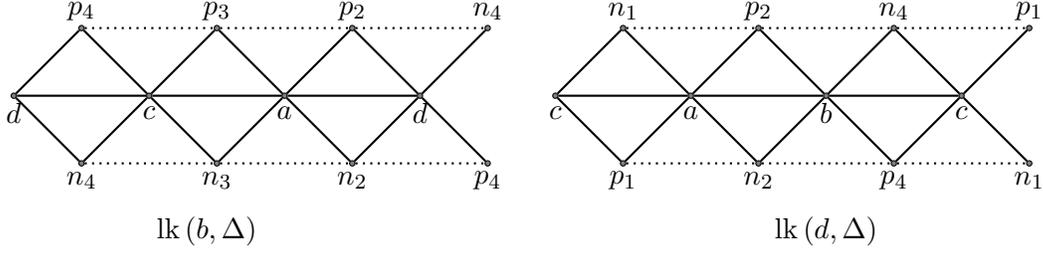
\begin{figure}[ht]
	\tikzstyle{ver}=[]
	\tikzstyle{vertex}=[circle, draw, fill=black!50, inner sep=0pt, minimum width=2pt]
	\tikzstyle{edge} = [draw,thick,-]
	\centering
	\begin{tikzpicture}[scale=0.45]
		
		\begin{scope}[shift={(-8,0)}]
			\foreach \x/\y/\z in {-6/1/b,-2/1/c,2/1/a,6/1/b1}{
				\node[vertex] (\z) at (\x,\y){};
			}
			
			\foreach \x/\y/\z in {-6/0.5/d,-2/0.5/c,2/0.5/a,6/0.5/d}{
				\node[ver] () at (\x,\y){$\z$};
			}
			
			\foreach \x/\y/\z in {-4/3/p_2,0/3/p_4,4/3/p_1,8/3/p2}{
				\node[vertex] (\z) at (\x,\y){};
			}
			
			\foreach \x/\y/\z in {-4/3.5/p_4,0/3.5/p_3,4/3.5/p_2,8/3.5/n_4}{
				\node[ver] () at (\x,\y){$\z$};
			}
			
			\foreach \x/\y/\z in {-4/-1/n_2,0/-1/n_4,4/-1/n_1,8/-1/q2}{
				\node[vertex] (\z) at (\x,\y){};
			}
			
			\foreach \x/\y/\z in {-4/-1.5/n_4,0/-1.5/n_3,4/-1.5/n_2,8/-1.5/p_4}{
				\node[ver] () at (\x,\y){$\z$};
			}

			\foreach \x/\y in {b/c,c/a,a/b1,b/p_2,c/p_2,b/n_2,c/n_2,c/p_4,a/p_4,c/n_4,a/n_4,a/p_1,b1/p_1,a/n_1,b1/n_1,
				b1/p2,b1/q2}{
				\path[edge] (\x) -- (\y);}
			
			\foreach \x/\y in {p_2/p_4,p_1/p2,n_2/n_4,n_1/q2}{
				\path[edge, dotted] (\x) -- (\y);}
			\path[edge, dotted] (p_4) -- (p_1);
			\path[edge, dotted] (n_4) -- (n_1);
			
		\end{scope}
		
		\begin{scope}[shift={(8,0)}]
			\foreach \x/\y/\z in {-6/1/b,-2/1/c,2/1/d,6/1/b1}{
				\node[vertex] (\z) at (\x,\y){};
			}
			
			\foreach \x/\y/\z in {-6/0.5/c,-2/0.5/a,2/0.5/b,6/0.5/c}{
				\node[ver] () at (\x,\y){$\z$};
			}
			
			\foreach \x/\y/\z in {-4/3/p_2,0/3/p_4,4/3/p_1,8/3/p2}{
				\node[vertex] (\z) at (\x,\y){};
			}
			
			\foreach \x/\y/\z in {-4/3.5/n_1,0/3.5/p_2,4/3.5/n_4,8/3.5/p_1}{
				\node[ver] () at (\x,\y){$\z$};
			}
			
			\foreach \x/\y/\z in {-4/-1/n_2,0/-1/n_4,4/-1/n_1,8/-1/q2}{
				\node[vertex] (\z) at (\x,\y){};
			}
			
			\foreach \x/\y/\z in {-4/-1.5/p_1,0/-1.5/n_2,4/-1.5/p_4,8/-1.5/n_1}{
				\node[ver] () at (\x,\y){$\z$};
			}

			\foreach \x/\y in {b/c,c/d,d/b1,b/p_2,c/p_2,b/n_2,c/n_2,c/p_4,d/p_4,c/n_4,d/n_4,d/p_1,b1/p_1,d/n_1,b1/n_1,
				b1/p2,b1/q2}{
				\path[edge] (\x) -- (\y);}
			
			\foreach \x/\y in {p_2/p_4,p_1/p2,n_2/n_4,n_1/q2}{
				\path[edge, dotted] (\x) -- (\y);}
			\path[edge, dotted] (p_4) -- (p_1);
			\path[edge, dotted] (n_4) -- (n_1);
			
		\end{scope}
		
		\node[ver] () at (-8.3,-3){$\lk (b,\Delta)$};
		\node[ver] () at (10,-3){$\lk (d,\Delta)$};
		
	\end{tikzpicture}
	\caption{$\lk (b,\Delta)$ and $\lk (d,\Delta)$ in $\Delta$.}\label{fig:lk(b&d)}
\end{figure}

\begin{proof}
	Let us consider $\lk (acd,\Delta)=\{p_1,n_1\}$, $\lk (abd,\Delta)=\{p_2,n_2\}$, $\lk (abc,\Delta)=\{p_3,n_3\}$ and $\lk (bcd,\Delta)$ $=\{p_4,n_4\}$ in $\Delta$. Given that, a small neighborhood of  $|\partial(bcd)|$ in $|\lk (a,\Delta)|$ is a M\"{o}bius strip. Since $\lk (c,\lk (a,\Delta))$ is a circle,  we have two possibilities   $C(d,p_1,\dots, p_3,b$, $n_3,\dots,n_1,d)$ or $C(d,p_1,\dots, n_3,b,p_3,\dots,n_1,d)$ for the circle. Without loss of generality we can assume that $C(d,p_1,\dots, p_3,b,n_3,\dots,n_1,d)$ is the circle. Since $\lk (b,\lk (a,\Delta))$ is a circle, without loss of generality we can assume that $C(c,p_3,\dots, p_2,d,n_2,\dots,n_3,c)$ is the circle. Again, $\lk (d,\lk (a,\Delta))$ is a circle. It is also known that a small neighborhood of  $|\partial(bcd)|$ in $|\lk (a,\Delta)|$ is a M\"{o}bius strip. Therefore, the circle $\lk (d,\lk (a,\Delta))$ should be of the form  $C(b,p_2,\dots, n_1,c,p_1,\dots,n_2,b)$.
	
	Next, we look at  $|\lk (c,\Delta)|$. We have already obtained the circle 
	$\lk (a,\lk (c,\Delta))$ = $\lk (c,\lk (a,\Delta))$ as $C(d,p_1,\dots, p_3,b,n_3,\dots,n_1,d)$. Without loss of generality we can assume that $C(a,p_3,\dots, p_4,d,n_4,\dots,n_3,a)$ is the circle $\lk (b,\lk (c,\Delta))$. Since a small neighborhood of  $|\partial(abd)|$ in $|\lk (c,\Delta)|$ is a M\"{o}bius strip, the circle $\lk (d,\lk (c,\Delta))$ should be of the form  $C(b,p_4,\dots, n_1,a,p_1,\dots,n_4,b)$.
	
	Now, we look at  $|\lk (b,\Delta)|$. We have already obtained the circle 
	$\lk (c,\lk (b,\Delta))$ = $\lk (b,\lk (c,\Delta))$ and  $\lk (a,\lk (b,\Delta))$ = $\lk (b,\lk (a,\Delta))$. Since a small neighborhood of  $|\partial(acd)|$ in $|\lk (b,\Delta)|$ is a M\"{o}bius strip, the circle $\lk (d,\lk (b,\Delta))$ should be of the form  $C(a,p_2,\dots, n_4$, $c,p_4,\dots,n_2,a)$. These give the circles  $\lk (a,\lk (d,\Delta))$, $\lk (b,\lk (d,\Delta))$ and $\lk (c,\lk (d,\Delta))$, and hence a portion of $\lk (d,\Delta)$ has been obtained in figure \ref{fig:lk(b&d)}. This gives  a small neighborhood of  $|\partial(abc)|$ in $|\lk (d,\Delta)|$ is a M\"{o}bius strip. This completes the proof.
\end{proof}

\begin{lemma}\label{lemma:tetrahedron}
Let $K$ be a normal $3$-pseudomanifold with $g_2(K)=4$ and $K$ has more than two singular vertices. Let $\partial(vxyz)$ be a missing tetrahedron in $K$ such that $\lk (v)$ is separated by the missing triangle $\partial(xyz)$. Then $K=K_1\# K_2$, where $g_2(K_2)=0$ and $vxyz \in K_1, K_2$. In other words, $K$ is obtained from a normal $3$-pseudomanifold $K_1$ by facet subdivisions.	
\end{lemma}

\begin{proof}
Since the link of a vertex $u\in K$ is either $\mathbb{S}^2$ or $\mathbb{RP}^2$, any missing triangle $\partial(\sigma)$ in $\lk (u)$ either separates $\lk (u)$ or a small neighborhood of $\partial(\sigma)$ in $\\lk (x)$ is a M\"{o}bius strip.
If $\partial(yzv)$ does not separate the $\lk (x)$ then a small neighborhood of $\partial(vyz)$ in $\lk (x)$ is a M\"{o}bius strip. Then from Lemmas \ref{lemma:three mobius} and \ref{lemma:missingtetra2}, we have link of exactly two vertices from $\{v,y,z\}$ are separated by the missing triangle formed by other three vertices. Thus Lemmas \ref{lemma:missingtetra1} and \ref{lemma:missingtetra2} imply that $K$ is obtained from a normal $3$-pseudomanifold $K'$ by handle addition, connected sum or an edge folding. 
If $K$ is obtained from $K'$ by handle addition then $g_2(K)=g_2(K')+10\geq 10$, which is a contradiction. If $K$ is obtained $K'$ by an edge folding then $g_2(K)=g_2(K')+3$. This implies $g_2(K')\leq 1$, and hence $K'$ is a sphere. Then $K$ has exactly two singular vertices, which is a contradiction. Therefore $K$ is obtained by connected sum of some normal 3-pseudomanifolds $K_1$ and $K_2$. If $K_1$ contains a singular vertex then $g_2(K_1)\geq3$. Then $g_2(K)=g_2(K_1)+g_2(K_2)$ implies that $g_2(K_2) \leq 1$. Since $K_2$ will not contain any singular vertices, $K_1$ has more than two singular vertices. Thus $g_2(K_1)=4$ and hence $g_2(K_2)=0$. Therefore, $|K|=|K_1|$
\end{proof}

Let $K$ be a normal $3$-pseudomanifold with $g_2(K)=4$ and $K$ has more than two singular vertices. If $\partial(vxyz)$ is a missing tetrahedron in $K$  such that $\lk (v)$ is separated by the missing triangle $\partial(xyz)$, then $K$ is obtained from $K_1$ by facet subdivisions. Thus we reduce the number of vertices. Further,  $g_2(K_1)=4$ and $K_1$ has more than two singular vertices. Thus, after finite number of steps, we get a normal  $3$-pseudomanifold $\tilde K$ such that $(i)$ $g_2(\tilde K)=4$, $(ii)$ $K$ is obtained from  $\tilde K$ by facet subdivisions, i.e., $|K|=|\tilde K|$ and $\tilde K$ has more than two singular vertices and $(iii)$ $\tilde K$ has no missing  tetrahedron  $\partial(abcd)$ where $\lk (x)$ is separated by the missing triangle $\partial(abcd-x)$, for all $x\in\{a,b,c,d\}$. From now onwards, we assume that the normal $3$-pseudomanifold $K$ has no missing  tetrahedron  $\partial(vxyz)$ where $\lk (v)$ is separated by the missing triangle $\partial(xyz)$.  Let $\mathcal{G}$ be the class of normal  $3$-pseudomanifolds $K$ such that $(i)$ $g_2(K)=4$, $(ii)$  $ K$ has more than two singular vertices and $(iii)$ $K$ has no missing  tetrahedron  $\partial(abcd)$ where $\lk (x)$ is separated by the missing triangle $\partial(abcd-x)$, for all $x\in\{a,b,c,d\}$.  Thus, we have the following result.

\begin{corollary}\label{corollary:class}
Let $K$ be a normal $3$-pseudomanifold with $g_2(K)=4$ and $K$ has more than two singular vertices. Then there is a normal $3$-pseudomanifold $\tilde K \in \mathcal{G}$ such that $K$ is obtained from the normal $3$-pseudomanifold $\tilde K$ by facet subdivisions.	
\end{corollary}

\begin{lemma}\label{lemma:reduction} If $K\in \mathcal{G}$ then there is no non-singular vertex in $K$ whose link contains a missing triangle.
\end{lemma}
\begin{proof} If possible let there be a non-singular vertex $v$ in $K$ and $\partial(xyz)\in\ \lk (v)$. We first observe that $xyz$ is a face of $K$. Otherwise, we could retriangulate the $3$-ball $\Star v$ by removing $v$, inserting $xyz$ and then coning off the two spheres formed by the link of $v$ and $xyz$. The resulting complex would have $g_2$ one less than $K$, that is $g_2(K)=3$ but it implies that $K$ has exactly two  singular vertices. This is a contradiction. Thus $xyz$ is in $K$
and $vxyz$ is a missing tetrahedron. Since $v$ is a non-singular vertex, $\lk (v)$ is separated by the missing triangle $\partial(xyz)$. This is again a contradiction as we assumed that the normal $3$-pseudomanifold $K$ has no such missing tetrahedron.
\end{proof}

\begin{lemma}\label{lemma:d>4}
	Let $K \in \mathcal{G}$. If $uv$ is an edge of $K$ then $d(uv)\geq 4$, i.e., $\lk (v,\lk (u))$ has at least four vertices.
\end{lemma}
 If possible let $d(uv)= 3$ and $\lk (uv)=\partial(abc)$. If possible, let $abc \in K$, then $\partial(uabc)$ and $\partial(vabc) \in K$. Then $\partial(abc)$ separates $\lk (u)$ and $\lk (v)$ in two parts. Since  $K \in \mathcal{G}$, we have $uabc,vabc\in K$. Thus $uabc$, $vabc$, $uvab$, $uvbc$, $uvac$ are all in $K$, i.e., $\partial(uvabc) \subset K$. But this is possible only if $\partial(uvabc) = K$ which is a contradiction as $K$ has more than two singular vertices. Thus, $abc \not \in K$. Let $K'=(K-\{\alpha\in K: uv \leq \alpha\})\cup\{abc, uabc, vabc\}$. Since $abc\not \in K$ and we retriangulate the 3-ball $\Star (uv)$, we have $|K'|\cong |K|$. Here $K'$ has also more than two singular vertices. Further, $f_0(K')=f_0(K)$ and $f_1(K')=f_1(K)-1$ and we get $g_2(K')=g_2(K)-1=3$. This implies  that $K'$ has exactly two singular vertices. This is again a contradiction. Thus, $d(uv)\geq 4$.

\begin{lemma}\label{lemma:intersection}
 Let $K\in \mathcal{G}$ and $a$ be a non-singular vertex in $K$ such that $ab\in K$. Then $\lk (a)\cap \lk (b)-\lk (ab)$ contains some vertices.
\end{lemma}
\begin{proof} 
If $\lk (a)\cap \lk (b) - \lk (ab)$ is an empty set then from Lemma \ref{homeomorphic} by contracting the edge $ab$ we get a normal $3$-pseudomanifold $K_1$ such that $|K| \cong |K_1|$ and $g_2(K_1)\leq 3$. Since $K_1$ has more than two singular vertices,
$g_2(K_1)$ can not be less than 4. Thus, $\lk (a)\cap \lk (b)-\lk (ab)$ is non-empty. If $\lk (a)\cap \lk (b)-\lk (ab)$ contains an open edge $(c,d)$ then $abcd$ is a missing tetrahedron in $K$. Since $K\in \mathcal{G}$ and  $\lk (a)$ is separated by the missing triangle $\partial(bcd)$, this is not possible. Thus,  $\lk (a)\cap \lk (b)-\lk (ab)$ contains some vertices.
\end{proof}

\begin{lemma}\label{lemma:two non-singular}
 If $K \in \mathcal {G}$ is a normal $3$-pseudomanifold with $g_2(K)=4$ and $K$ has $8$ singular vertices then $|V(K)|\geq 10$.
\end{lemma}
\begin{proof}
It follows from Lemma \ref{lemma:lower bound} that the only possible link of the singular vertices in $K$ is $\mathbb{RP}^2$. Thus $d(u)\geq 6$ for any singular vertex $u$ in $K$. Further, Lemmas \ref{lemma:d>4} and \ref{lemma:intersection} imply $d(v)\geq 6$ for any non-singular vertex $v$ in $K$. Now, we prove that it is not possible that $|V(K)|=8$ and $K$ has $8$ singular vertices. If possible, let  $|V(K)|=8$ and $K$ have $8$ singular vertices. Then we must have $4$ singular vertices of degree $6$ and $4$ singular vertices of degree $7$. Let $V(K)=\{v_1, \dots ,v_8\}$ and $d(v_1)=6$ and $v_8\not\in \lk (v_1)$. Then the degree of $v_8$ must be $6$ and in this case all other $6$ vertices except $v_1$ and $v_8$ will have degree $7$, which is a contradiction.
	
Now, we prove that $V(K)=9$ is also not possible. If this is possible then there must be a non-singular vertex in $K$. Let $V(K) = \{v_1, \dots ,v_8,v\}$ where $v$ is the only non-singular vertex. Since $g_2(K)=4$ and $|V(K)|=9$, we have $|E(K)|=30$.
	
\noindent \textbf{Case 1:}  Let there be no singular vertex of degree $6$. Let $X$ be the number of singular vertices of degree $7$ and $Y$ be the number of singular vertices of degree $8$. Then we have
	$d(v) + 7X + 8Y = 60$ and $X + Y = 8$. Solving these two we get that $d(v) + 4 = X$. But $d(v)\geq 6$ implies $X\geq 10$, which is a contradiction.
		
\noindent \textbf{Case 2:} Let $v_1$ be a singular vertex of degree $6$ and $V(\lk (v_1))=\{v_2,v_3, \dots ,v_6,v_7\}$ then $v_8,v\not\in\ \lk (v_1)$. Since $d(v),d(v_8)\geq 6$, at least $4$ vertices of $\lk (v_1)$ must be of degree $8$ and other two must be of degree at least $7$. Then $\sum_{v\in V(K)}d(v)\geq 64 >60$, which is a contradiction.
\end{proof}

 \begin{lemma}\label{lemma:degree 8}
 If $K \in \mathcal {G}$ then for any non-singular vertex $a\in K$, $d(a)\leq 8$.
 \end{lemma}
 
 \begin{proof} 

 If possible let $d(a)\geq 9$. Then by Lemma \ref {lemma:intersection}, for any  vertex $u\in \lk (a)$, $\lk (a)\cap \lk (u) - \lk (ua)$ contains a vertex, say $x$. Thus, $ux\not\in \lk (a)$. Since $d(a)\geq 9$, we have at least $5$ such edges, say $e_1,e_2, \dots ,e_5$ in $K$. Since the edges  $e_1,e_2, \dots ,e_5\not\in \lk(a)$ but the vertices of  $e_1,e_2, \dots ,e_5$ are in $\lk(a)$, by Corollary \ref{missingedges} we have $g_2(K)\geq g_2(\lk(a))+5= 5$, which is a contradiction.
 \end{proof}

\begin{lemma}\label{lemma:no edge}
 If $K \in \mathcal {G}$ then there is no edge between two non-singular vertices.
\end{lemma}
\begin{proof}
If possible, let $a,b$ be two non-singular vertices such that $ab\in K$. Then $\lk (ab)=C_n(u_1, \dots ,u_n)$ where $n\geq 4$. If one of $u_1,u_2, \dots ,u_n$ is singular, say $u_1$ then $g_2(u_1)=3$ and $ab\in \lk (u_1)$.  Then from Lemma \ref {lemma:intersection}, $\lk (u_1)\cap \lk (a) - \lk (au_1)$ contains a vertex, say $x$ and $\lk (b)\cap \lk (u_1) - \lk (bu_1)$ contains a vertex, say $y$. Then $ax,by\not\in \lk (u_1)$ but $a,x,b,y\in \lk (u_1)$.
Thus, by Corollary \ref{missingedges} we have $g_2(K)\geq g_2(\lk(u_1))+2=5$, which is a contradiction. Thus $u_1, \dots ,u_n$ are non-singular vertices. 

Further, from Lemma \ref{lemma:degree 8}, we have $d(a),d(b)\leq 8$. Let $\lk (a)\cap\lk (b) - \lk (ab)$ contains some vertex, say $z$. Then $d(az),d(bz)\geq 4$ and $z$ is joined with at least two $u_i$ in both $\lk (a)$ and $\lk (b)$. But $zu_i$ can not be common in $\lk (a)\cap \lk (b)$, otherwise $abzu_i$ is a missing tetrahedron with some non-singular vertices, which is not possible. 

Let $zu_j,zu_k\in \lk (b)$. Then $zu_j,zu_k\not \in \lk (a)$. Thus we have  $zu_j,zu_k, bz\not \in \lk (a)$. There are four more vertices in $\lk (a)$ other than $b,z,u_j,u_k$. Then from Lemma \ref{lemma:intersection}, we have $\lk (a)\cap \lk (x) - \lk (ax)$ contains a vertex, for all $x\in V(\lk (a))\setminus \{b,z,u_j,u_k\}$. Thus, we have at least two edges $e_1,e_2$ other than $zu_j,zu_k, bz$ whose end points are in $\lk (a)$ but $e_1, e_2\not \in \lk (a)$. Therefore,  we have 5 edges $e_1,e_2,zu_j,zu_k,bz\not\in \lk(a)$ but the vertices of  $e_1,e_2,zu_j,zu_k,bz$ are in $\lk(a)$. Thus, by Corollary \ref{missingedges} we have $g_2(K)\geq g_2(\lk(a))+5=5$, which is a contradiction. Hence there is no edge between two non-singular vertices.
\end{proof}

\begin{lemma}\label{lemma:one non-singular}
 If $K \in \mathcal {G}$  then every singular vertex can be connected with at most one non-singular vertex.
\end{lemma}

\begin{proof} 
Let $t\in\ K$ be a singular vertex. If possible let $a,b\in \lk (t)$ be two non-singular vertices. From Lemma \ref{lemma:no edge}, we have $ab\not\in K$. Further, by Lemma \ref {lemma:intersection}, $\lk (a)\cap \lk (t) - \lk (at)$ contains a vertex, say $x$ and $\lk (b)\cap \lk (t) - \lk (bt)$ contains a vertex, say $y$. Then $ax,by \not \in \lk (t)$ and $ax,by$ are two different edges. Then, by Corollary \ref{missingedges} we have $g_2(K)\geq g_2(\lk(t))+2= 5$, which is a contradiction. Thus,  every singular vertex can be connected with at most one non-singular vertex.

\end{proof}

\begin{lemma}\label{lemma:4 and 8} 
There does not exist any normal $3$-pseudomanifold with more than $2$ singular vertices and $g_2(K)=4$.
\end{lemma}

\begin{proof}

Let $K$ be a normal $3$-pseudomanifold with $g_2(K)=4$ and $K$ has more than two singular vertices. Then by Corollary \ref{corollary:class} there is a normal $3$-pseudomanifold $\tilde K \in \mathcal{G}$ such that $K$ is obtained from the normal $3$-pseudomanifold $\tilde K$ by facet subdivisions. Now, we can proceed with $\tilde K$. If $\tilde K$ has exactly $4$ singular vertices then there must be an edge between two non-singular vertices but this is not possible. Therefore there does not exist any normal $3$-pseudomanifold with $g_2(K)=4$ and with exactly $4$ singular vertices. 

Now if $\tilde K$ has exactly $8$ singular vertices then from Lemma \ref{lemma:two non-singular}, there must be at least two non-singular vertices. Thus there must be an edge between a singular vertex say $t$ and a non-singular vertex say $u$. Since $d(t)\geq 6$, $d(tu)\geq 4$ and $\lk (t)\cap \lk (u)-\lk (tu)$ contains some singular vertex, from Lemmas \ref{lemma:no edge} and  \ref{lemma:one non-singular}, we have $u$ is adjacent to at least $6$ singular vertices say $t_1, \dots ,t_6$. Let $t_7$ and $t_8$ be other two singular vertices. Let $v$ is  another non-singular vertex. Then from Lemmas \ref{lemma:no edge} and  \ref{lemma:one non-singular}, $v$ is only adjacent with $t_7$ and $t_8$. Thus $d(v)=2$, which is  a contradiction. Therefore there does not exist any normal $3$-pseudomanifold with $g_2(K)=4$ and with exactly $8$ singular vertices. This completes the proof.
 \end{proof}
 
 \begin{lemma}\label{lemma:two singular}
 If $K$ is a normal $3$-pseudomanifold with $g_2(K)=4$ and $K$ has exactly two singular vertices then $K$ is obtained from some boundary of $4$-simplices  by a sequence of operations such as connected sum, edge expansion and an edge folding.
 \end{lemma}
\begin{proof} 
Let $K$ have a missing tetrahedron $\sigma =\partial(abcd)$ such that for every vertex $x\in\sigma$, $\sigma-x$ separates the link of $x$. Then from Lemma \ref{lemma:missingtetra1}, $K$ is obtained by handle addition or connected sum. If $K$ is obtained from $K'$ by a handle addition then $g_2(K)=g_2(K')+10\geq 10$, which is not possible. Thus, $K$ is obtained from some connected sums, i.e., $K=K_1\#K_2$. If $K_1$ contains a singular vertex then $K_1$ contains the other singular vertex as well. Thus $g_2(K_1)\geq 3$ and $g_2(K_2)\leq 1$.   
  So, there is a finite collection of normal $3$-pseudomanifolds $K_1, \dots ,K_n$ such that $K=K_1\# \dots \# K_n$ where $K_1$ contains both the singular vertices. Further for $2\leq i \leq n$,  $K_i$ has no missing tetrahedron, and if $K_1$ has a missing tetrahedron $\sigma$ then $\sigma=uvab$, where $u,v$ are the singular vertices and a small neighbourhood of $\partial (vab)$ in $\lk (u)$ is a M\"obius strip. If $g_2(K_i)=1$ for some $2\leq i \leq n$, then by Lemma \ref{lemma:g2<3}, $K_i$ is obtained from  boundary of $4$-simplices  by a sequence of operations such as connected sum and edge expansion.  If $g_2(K_i)=0$ then  $K_i$ is a connected sum  of boundary of $4$-simplices. Now, we look at $K_1$.
 
 \smallskip
  
  \noindent \textbf{Case 1:}  Let $K_1$ have a missing tetrahedron $uvab$,  where $u,v$ are the singular vertices and a small neighbourhood of $\partial (vab)$ in $\lk (u)$ is a M\"obius strip. Then by Lemma \ref{lemma:missingtetra2}, $K_1$ is obtained from $K_1'$ by an edge folding and $g_2(K_1')=g_2(K)-3\leq 1$. Thus,  by Lemma \ref{lemma:g2<3} $K_1'$ is obtained from  boundary of $4$-simplices  by a sequence of operations such as connected sum and edge expansion. Then we are done.
  
  \smallskip
  
  \noindent \textbf{Case 2:} Let $K_1$ have no missing tetrahedron. Let $u$ be the singular vertex. The $d(u)\geq 6$. Let $x$ be a non-singular vertex such that $d(xu)=3$, say $\lk (xu)=\partial(pqr)$. If   $pqr \in K_1$ then $\partial (xpqr), \partial (upqr)  \in K_1$.
  Since $K_1$ have no missing tetrahedron, $xpqr, upqr  \in K_1$. Thus, $K_1=\partial(uxpqr)$. This is a contradiction. Thus, $pqr \not \in K_1$. Thus, $p,q,r$ has degree at least 4 in $\lk (u)$. Thus we have always at least two non-singular vertices say $p,q \in \lk (u)$ such that $pq \in \lk (u)$ and $d(pu),d(qu)\geq 4$.  If $\lk (u)\cap \lk (p) - \lk (up)$ contains an open edge $(a,b)$ then $\partial(upab)$ is a missing tetrahedron, which is a contradiction. If $\lk (u)\cap \lk (p) - \lk (up)$ contains some vertex, say $x$ then $x\neq q$ and $px\not\in \lk (u)$. Similarly,  if $\lk (u)\cap \lk (q) - \lk (uq)$ contains some vertex, say $y$ then $y\neq p$ and $qy\not\in \lk (u)$. Then $px,qy \in K_1$ where $p,q,x,y\in \lk (u)$ but $px,qy \not \in \lk (u)$. Then  from Corollary \ref{missingedges}, we have $g_2(K_1) \geq g_2(\Star (u))+2 = 5$. This is a contradiction. Thus,  $\lk (u)\cap \lk (p) - \lk (up)$ or $\lk (u)\cap \lk (q) - \lk (uq)$ must be empty. Without loss of generality, assume that  $\lk (u)\cap \lk (p) - \lk (up)$  is empty. Then we contract the edge $up$ and we get a new normal $3$-pseudomanifold $\tilde K_1$. From Lemma \ref{homeomorphic}, we have $|\tilde K_1|\cong |K_1|$ and $g_2(\tilde K_1) \leq g_2(K_1)-1$. Thus, if such edge contraction is possible then $g_2(K_1)=4$ and $g_2(\tilde K_1)=3$, and $\tilde K_1$ may have a missing tetrahedron. Now we repeat the same arguments from the beginning. Since   $g_2(\tilde K_1)$ can not be reduced further, there will be no more edge contraction. Thus we must move to Case 1 only.
  
Thus, $K$ is obtained from a boundary of $4$-simplices  by a sequence of operations such as connected sum, edge expansion and  edge folding.
 \end{proof}
 
  Let $K$ be a normal $3$-pseudomanifold with $g_2(K)\leq 4$, which is not a sphere. Then by Lemma \ref{lemma:4 and 8}, $K$ has exactly two singular vertices. Then combining with Lemma \ref{lemma:two singular}, we have the following result.
  
  \begin{theorem}
   If $K$ is a normal $3$-pseudomanifold with $g_2(K)\leq 4$, which is not a sphere then $K$ is obtained from some boundary of $4$-simplices  by a sequence of operations such as connected sum, edge expansion and an edge folding.
 \end{theorem}
 
\medskip
 
  \noindent {\bf Acknowledgement:} 
The first author is supported by DST INSPIRE Faculty Research Grant (DST/INSPIRE/04/2017/002471). The second author is supported by CSIR (India).

	 {

\end{document}
\begin{thebibliography}{999}
			
			\bibitem{BagchiDatta} B.~Bagchi and B.~Datta, Lower bound theorem for normal pseudomanifolds, {\em Expo. Math.} 26 (2008), 327--351.
			
			\bibitem{BagchiDatta98}  B.~Bagchi and B.~Datta, A structure theorem for pseudomanifolds, {\em Discrete Math.} 188 (1998), no. 1-3, 41--60.
			
			\bibitem{Barnette1} D.~Barnette, A proof of the lower bound conjecture for convex polytopes, {\em Pacific J. Math.} 46 (1973), 349--354.
			
			\bibitem{Barnette2} D.~Barnette, Graph theorems for manifolds, {\em Israel J. Math.} 16 (1973), 62--72.
			
			\bibitem{Barnette3} D.~Barnette,  The minimum number of vertices of a simple polytope, {\em Israel J. Math.} 10 (1971) 121--125. 
			
			\bibitem{BSR} B.~Basak, R.K.~Gupta, S.~Sarkar,  A characterization of normal 3-pseudomanifolds with at most two singularities, arXive:2104.03751.
			
			\bibitem{BasakSwartz}  B. Basak and E. Swartz, Three-dimensional normal pseudomanifolds with relatively few edges, {\em Advances in Mathematics} 365 (2020) 107035, 1--25.
			
			
			\bibitem{Fogelsanger}
			A.~Fogelsanger, The generic rigidity of minimal cycles, 
			Ph.D.~thesis, Cornell University, 1988.
			
			\bibitem{Gromov} M.~Gromov, {\em Partial differential relations}, Springer, Berlin Heidelberg New York, 1986.
			
			
			\bibitem{Kalai} G.~Kalai, Rigidity and the lower bound theorem I, 
			{\em Invent.~Math.} 88 (1987), 125--151. 
			
			
			\bibitem{NevoNovinsky}
			E.~Nevo and E.~Novinsky, A characterization of simplicial polytopes with $g_2=1$, {\em J. Combin. Theory Ser. A} 118 (2011), 387--395.
			
			
			\bibitem{NovikSwartz}
			I.~Novik and E.~Swartz, Face numbers of pseudomanifolds with isolated singularities, {\em Math. Scan.} 110 (2012), 198--212.
			
			\bibitem{Swartz2008} E.~Swartz, Topological finiteness for edge-vertex enumeration, {\em Adv. Math.}, 219 (2008), 1722--1728.
			
			
			\bibitem{Swartz2009}
			E.~Swartz, Face enumeration:  From spheres to manifolds, {\em J. Eur. Math. Soc.}, 11 (2009), 449--485.
			
			\bibitem{Swartzcounting} E.~Swartz, Thirty-five years and counting, arXive:1411.0987.
			
			\bibitem{TayWhiteWhiteley}
			T.~Tay, N.~White and W.~Whiteley, Skeletal rigidity of simplicial complexes II, European J. Combin., 16 (1995), 503--525.
			
			\bibitem{Walkup}
			D.~Walkup,
			The lower bound conjecture for $3$- and $4$-manifolds,
			Acta~Math.~125 (1970), 75--107.
			
			
			\bibitem{Zheng}
			H.~Zheng, A characterization of homology manifolds with $g_2\le 2$, {\em J. Combin. Theory Ser. A} 153 (2018), 31--45.
		\end{thebibliography}
